\documentclass[12pt]{amsart}
\usepackage[margin=1.2in]{geometry}
\usepackage{graphicx,latexsym}
\usepackage{comment}
\usepackage{amsfonts}
\usepackage{amsmath}
\usepackage{amssymb,amsxtra,amscd,amsthm}
\usepackage{mathrsfs}
\usepackage{amsfonts, amssymb, amsmath, amsthm, bm, hyperref}
\usepackage{tikz}

\newcommand{\N}{\mathbb{N}}
\newcommand{\Z}{\mathbb{Z}}

\newcommand{\C}{\mathbb{C}}

\newtheorem{theorem}{Theorem}[section]
\newtheorem{proposition}[theorem]{Proposition}
\newtheorem{lemma}[theorem]{Lemma}
\newtheorem{corollary}[theorem]{Corollary}

\theoremstyle{definition}

\theoremstyle{remark}
\newtheorem{remark}[theorem]{Remark}

\numberwithin{equation}{section}
\usetikzlibrary{arrows,chains,matrix,positioning,scopes}
\makeatletter
\tikzset{join/.code=\tikzset{after node path={%
\ifx\tikzchainprevious\pgfutil@empty\else(\tikzchainprevious)%
edge[every join]#1(\tikzchaincurrent)\fi}}}
\makeatother
\tikzset{>=stealth',every on chain/.append style={join},
         every join/.style={->}}
\tikzstyle{labeled}=[execute at begin node=$\scriptstyle,
   execute at end node=$]

\begin{document}
\title{The Borel-Ritt problem in Beurling ultraholomorphic classes}

\author[A. Debrouwere]{Andreas Debrouwere}
\thanks{A. Debrouwere was supported by  FWO-Vlaanderen through the postdoctoral grant 12T0519N}
\address{A. Debrouwere, Department of Mathematics: Analysis, Logic and Discrete Mathematics\\ Ghent University\\ Krijgslaan 281\\ 9000 Gent\\ Belgium}
\email{andreas.debrouwere@UGent.be}


\begin{abstract}
We give a complete solution to the Borel-Ritt problem in non-uniform spaces $\mathscr{A}^-_{(M)}(S)$ of ultraholomorphic functions of Beurling type, where $S$ is  an unbounded sector  of the Riemann surface of the logarithm and $M$ is a strongly regular weight sequence. Namely, we characterize the surjectivity and the existence of a continuous linear right inverse of the asymptotic Borel map on $\mathscr{A}^-_{(M)}(S)$ in terms of the aperture of the sector $S$ and the weight sequence $M$. Our work improves previous results by  Thilliez \cite{Thilliez} and Schmets and Valdivia \cite{S-V}.
\end{abstract}
\maketitle

\section{Introduction}
The Borel-Ritt problem for spaces of ultraholomorphic functions goes back to the seminal work of Ramis \cite{Ramis} and plays an important role  in the study of formal power series solutions to various kinds of equations \cite{Balser}; see \cite{Debrouwere-Stieltjes,JG-S-S, JG-S-S-2, S-V,Thilliez} for an account of known results and open problems. In the present article, we consider this problem in the setting of non-uniform Beurling ultraholomorphic classes.

Let $M = (M_{p})_{p \in \N}$ be a sequence of positive real numbers and fix a sector of the Riemann surface $\Sigma$ of the logarithm
$$
S_\gamma = \{ z \in \Sigma \, | \, | \operatorname{Arg} z | < \frac{\gamma \pi}{2} \}, \qquad \gamma > 0.
$$
We define $\mathscr{A}_{(M)}(S_\gamma)$ as the Fr\'echet space consisting of all functions $f$ holomorphic in $S_\gamma$ such that
$$
\sup_{p \in \N} \sup_{z \in S_\gamma} \frac{n^p |f^{(p)}(z)|}{p!M_p} < \infty, \qquad \forall n \in \N,
$$
and set
$$
\mathscr{A}^-_{(M)}(S_\gamma) := \varprojlim_{\delta \to \gamma^-} \mathscr{A}_{(M)}(S_\delta).
$$
We call $\mathscr{A}_{(M)}(S_\gamma)$ and $\mathscr{A}^-_{(M)}(S_\gamma)$ \emph{the uniform and non-uniform space of ultraholomoprhic functions of class $(M)$ (of Beurling type) in $S_\gamma$}, respectively.  The corresponding sequence  space $\Lambda_{(M)}(\N)$ is defined as the Fr\'echet space consisting of all $(c_p)_{p \in \N} \in \C^\N$ such that
$$
\sup_{p \in \N}  \frac{n^p |c_p|}{p!M_p} < \infty,  \qquad \forall n \in \N.
$$
Consider the \emph{asymptotic Borel map}  
\begin{equation}
B: \mathscr{A}_{(M)}(S_\gamma) \rightarrow \Lambda_{(M)}(\N), \, f \mapsto (f^{(p)}(0))_{p \in \N},
\label{borel-mapping}
\end{equation}
where $f^{(p)}(0) := \lim_{\substack{z \to 0^+ \\ z \in S_{\gamma}}} f^{(p)}(z)$. We also define
\begin{equation}
B: \mathscr{A}^-_{(M)}(S_\gamma) \rightarrow \Lambda_{(M)}(\N), \, f \mapsto (f^{(p)}(0))_{p \in \N}, 
\label{borel-mapping-nu}
\end{equation}
where now $f^{(p)}(0) = \lim_{\substack{z \to 0^+ \\ z \in S_{\delta}}} f^{(p)}(z)$ for each $\delta < \gamma$.
Obviously, these maps are well-defined and continuous. In this context, the Borel-Ritt problem consists in finding conditions on $\gamma$  and $M$ which ensure that the maps  \eqref{borel-mapping} and \eqref{borel-mapping-nu} are surjective.  If this is the case, it is natural to ask in addition whether these maps have a continuous linear right inverse.

In 2003 Thilliez showed the following fundamental result; see Section \ref{sect-prelim} for unexplained notions about weight sequences.
\begin{theorem} \label{theorem-T} \cite[Corollary 3.4.1]{Thilliez} Let $M$ be a strongly regular weight sequence and let $0 < \gamma < \gamma(M)$. Then, the map \eqref{borel-mapping} is surjective. 
\end{theorem}
Schmets and Valdivia \cite{S-V}  already proved in 2000 the surjectivity of the map \eqref{borel-mapping} under the following more restrictive condition: $\gamma < n < \gamma(M)$ for some $n \in \Z_+$. But, in contrast to Thilliez, they also showed that the map \eqref{borel-mapping} has a continuous linear right inverse. Theorem \ref{theorem-T} is  believed to be optimal in the sense that the condition $\gamma < \gamma(M)$ is also necessary for the map \eqref{borel-mapping} to be surjective. However, for general strongly regular weight sequences, it is only known that the surjectivity of the map \eqref{borel-mapping} implies $\gamma \leq \gamma(M)$ (cf.\ \cite{JG-S-S})\footnote{This result is shown in \cite[Corollary 4.19]{JG-S-S} in the Roumieu case but, as we shall explain in Section \ref{sect-main}, the proof provided there can be adapted to the Beurling case.}.

In this article we give a complete solution to the Borel-Ritt problem in the non-uniform spaces $\mathscr{A}^-_{(M)}(S_\gamma)$. More precisely, we show the following result.
\begin{theorem} \label{intro}  Let $M$ be a strongly regular weight sequence and let $\gamma > 0$. Then, the following statements are equivalent:
\begin{itemize}
\item[$(i)$] $\gamma \leq \gamma(M)$.
\item[$(ii)$] $B: \mathscr{A}^-_{(M)}(S_\gamma) \rightarrow \Lambda_{(M)}(\N)$ is surjective.
\item[$(iii)$] $B: \mathscr{A}^-_{(M)}(S_\gamma) \rightarrow \Lambda_{(M)}(\N)$ has a continuous linear right inverse.
\end{itemize}
\end{theorem}
The implication $(ii) \Rightarrow (i)$ follows  from the analogous statement for the uniform spaces $\mathscr{A}_{(M)}(S_\gamma)$ mentioned above, whereas the implication $(i) \Rightarrow (ii)$ for $\gamma < \gamma(M)$ is a consequence of Theorem \ref{theorem-T}. The main novelty of Theorem \ref{intro}  comprises of  the following two results: The map \eqref{borel-mapping-nu} is surjective in the limit case $\gamma = \gamma(M)$; the map \eqref{borel-mapping-nu} is surjective if and only if it has a continuous linear right inverse. Furthermore,
Theorem \ref{intro} yields the following corollary for the uniform spaces $\mathscr{A}_{(M)}(S_\gamma)$; it simultaneously improves the results of Thiliez \cite{Thilliez} (Theorem \ref{theorem-T})  and Schmets and Valdivia \cite{S-V} mentioned above.
\begin{corollary}Let $M$ be a strongly regular weight sequence and let $0 < \gamma < \gamma(M)$. Then, the map \eqref{borel-mapping} has a continuous linear right inverse. 
\end{corollary}
Our proof of Theorem \ref{intro} is based on two classical results from the general theory of Fr\'echet spaces, namely, the abstract Mittag-Leffler Lemma \cite{Wengenroth} and the $(DN)$-$(\Omega)$ splitting theorem \cite{Vogt}.  We explain these results in Section \ref{sect-aux1}. Moreover, we make use of certain linear topological properties of spaces of holomorphic functions with rapid decay on strips, recently obtained by the author in  \cite{Debrouwere}. These spaces and their properties are reviewed in Section \ref{sect-aux2}. We also show there how these spaces  are connected to the Borel-Ritt problem for Beurling ultraholomorphic classes. Finally, the proof of our main result Theorem \ref{intro}  is given in Section \ref{sect-main}. We would like  to point out that the technical core of the proof of Theorem \ref{intro} is contained in the results we use from \cite{Debrouwere}.

\section{Beurling ultraholomorphic classes}\label{sect-prelim}
In this preliminary section we introduce various spaces of ultraholomorphic functions of Beurling type that are used throughout this paper. Furthermore, we state our main result.

We start by discussing weight sequences. A sequence $M = (M_{p})_{p \in \N}$ of positive real  numbers is called a \emph{weight sequence} if $\lim_{p \to \infty} M_{p}^{1/p} =  \infty$. We set $m_p = M_p/M_{p-1}$ for $p \in \Z_+$. 
We consider the following conditions on a weight sequence $M$:
\begin{itemize}
\item[$(M.1)$] \emph{(log-convexity)}  $M^2_{p} \leq M_{p-1}M_{p+1}$, $p \in \Z_+$.
\item[$(M.2)$]  \emph{(moderate growth)} $M_{p+q} \leq CL^{p+q} M_{p} M_{q}$, $p,q\in \N$, for some $C,L>0$.
\item[$(M.3)$]   \emph{(strong non-quasianalyticity)}$\displaystyle  \sum_{q = p+1}^\infty \frac{1}{m_q} \leq  \frac{Cp}{m_{p+1}}$, $p \in \Z_+$, for some $C > 0$.
\end{itemize}
We refer to  \cite{Komatsu} for the meaning of these conditions. Following Thilliez \cite{Thilliez}, we call a weight sequence $M$ \emph{strongly regular} if $M$ satisfies $(M.1)$ and $(M.2)$, and $(p!M_p)_{p \in \N}$ satisfies $(M.3)$.  The most important examples of strongly regular weight sequences are the \emph{Gevrey sequences} $p!^{\alpha}$, $\alpha>0$. 

Let $M$ be a weight sequence and let $\gamma > 0$. We say that $M$ satisfies property $(P_\gamma)$ if there exists a sequence $(n_p)_{p \in \Z_+}$ of positive real numbers such that $(n_p/p^\gamma)_{p \in \Z_+}$ is non-decreasing and there is $C > 0$ such that $C^{-1}n_p \leq m_p \leq Cn_p$ for all $p \in \Z_+$. We define the \emph{growth index} $\gamma(M)$ of $M$ as \cite[Definition 1.3.5]{Thilliez}
$$
\gamma(M) := \sup \{\gamma > 0 \, | \, M \mbox{ satisfies } (P_\gamma) \}.
$$
If $M$ is strongly regular, then $0 <\gamma(M) < \infty $ \cite[Lemma 1.3.2]{Thilliez}. 

The \emph{associated function} of a weight sequence $M$ is defined as
$$
\omega_M(t) := \sup_{p\in\N}\log \frac{t^pM_0}{M_p},\qquad t \geq 0.
$$
Suppose that $M$ satisfies $(M.1)$. Then, $M$ satisfies $(M.2)$ if and only if  \cite[Proposition 3.6]{Komatsu}
\begin{equation}
2\omega_M(t) \leq \omega_M(Lt) +  C, \qquad t \geq 0,
\label{assM2}
\end{equation}
for some $C,L >0$, whereas $\gamma(M) > 0$ implies that \cite[Corollary 2.14 and Corollary 4.6$(i)$]{JG-S-S-ind}
\begin{equation}
\omega_M(2t) \leq L\omega_M(t) + C, \qquad t \geq 0,
\label{assgamma}
\end{equation}
for some $C,L > 0$. 

The relation $M \prec N$ between two weight sequences means that  for all $L > 0$ there is $C > 0$ such that
$M_p\leq CL^{-p}N_{p}$ for all $p\in\N$. By \cite[Lemma 3.10]{Komatsu}, we have that $M \prec N$ if and only if for all 
$$
\omega_N(Lt) \leq \omega_M(t) +  C, \qquad t \geq 0,
$$
for all $L > 0$ and suitable $C = C_L > 0$.

Let $M$ be a weight sequence and let $s > 0$. We define $M^s = (M^s_p)_{p \in \N}$. Note that $\omega_{M^s}(t) = s \omega_M(t^{1/s})$ for $t \geq 0$ and $\gamma(M^s) = s \gamma(M)$. Furthermore, $M^s$ is strongly regular if $M$ is so \cite[Lemma 1.3.4]{Thilliez}.

We  need the following characterization of $\gamma(M)$  \cite{JG-S-S-ind}.
\begin{lemma} \label{gammaM3}
Let $M$ be a strongly regular weight sequence  and let $\gamma > 0$. Then, the following statements are equivalent:
\begin{itemize}
\item[$(i)$] $\gamma < \gamma(M)$.
\item[$(ii)$] $M^{1/\gamma}$ satisfies $(M.3)$.
\item[$(iii)$] There are $C,L > 0$ such that
$$
\int_1^\infty \frac{\omega_M(ts)}{s^{1+ (1/\gamma)}} ds \leq L \omega_M(t) + C, \qquad t \geq 0. 
$$
\end{itemize}
\end{lemma}
\begin{proof}
$(i) \Leftrightarrow (ii)$ This is shown in \cite[Corollary 3.12$(iii)$]{JG-S-S-ind}.

$(i) \Leftrightarrow (iii)$ This follows from \cite[Theorem 2.11 and Corollary 4.6$(iii)$]{JG-S-S-ind}.
\end{proof}
Next, we introduce various spaces of ultraholomorphic functions of Beurling type. We denote by $\Sigma$ the Riemann surface of the logarithm. Given an open subset $\Omega$ of $\Sigma$ (or  of the complex plane $\C$), we write $\mathscr{O}(\Omega)$ for the space of holomorphic functions in $\Omega$.  We define the sectors
$$
S_\gamma := \{ z \in \Sigma \, | \, | \operatorname{Arg} z | < \frac{\gamma \pi}{2} \}, \qquad \gamma > 0.
$$

Let $M$ be a weight sequence and let $\gamma >0$. We define $\mathscr{A}_{(M)}(S_\gamma)$ as the Fr\'echet space consisting of all $f \in \mathscr{O}(S_\gamma)$ such that
$$
\sup_{p \in \N} \sup_{z \in S_\gamma} \frac{n^p |f^{(p)}(z)|}{p!M_p} < \infty, \qquad \forall n \in \N,
$$
and $\mathscr{S}_{(M)}(S_\gamma)$ as the Fr\'echet space consisting of all $f \in \mathscr{O}(S_\gamma)$ such that
$$
\sup_{p \in \N} \sup_{z \in S_\gamma} \frac{n^p |f^{(p)}(z)|e^{\omega_M(n|z|)}}{p!M_p} < \infty, \qquad \forall n \in \N.
$$
We are mainly interested in the following non-uniform spaces
$$
\mathscr{A}^{-}_{(M)}(S_\gamma) := \varprojlim_{\delta \rightarrow \gamma^-} \mathscr{A}_{(M)}(S_\delta), \qquad \mathscr{S}^{-}_{(M)}(S_\gamma) := \varprojlim_{\delta \rightarrow \gamma^-} \mathscr{S}_{(M)}(S_\delta).
$$
Note that $\mathscr{S}_{(M)}(S_\gamma) \subset \mathscr{A}_{(M)}(S_\gamma)$ and $\mathscr{S}^{-}_{(M)}(S_\gamma) \subset \mathscr{A}^{-}_{(M)}(S_\gamma)$  continuously. 
We define $\Lambda_{(M)}(\N)$ as the Fr\'echet space consisting of all sequences $(c_p)_{p \in \N} \in \C^\N$ such that
$$
 \sup_{p \in \N} \frac{n^p |c_p|}{p!M_p} < \infty, \qquad \forall n \in \N.
$$ 
The \emph{asymptotic Borel map} is defined as
$$
B: \mathscr{A}^{-}_{(M)}(S_\gamma) \rightarrow \Lambda_{(M)}(\N), \, f \mapsto  (f^{(p)}(0))_{p \in \N},
$$
where 
$$
f^{(p)}(0) := \lim_{\substack{z \to 0^+ \\ z \in S_{\delta}}} f^{(p)}(z) \in \C, \qquad p \in \N,
$$
exist for each $\delta < \gamma$. Obviously, this map is well-defined and continuous. We  also use the notation $B$ for the restriction of the asymptotic Borel map to various subspaces of $\mathscr{A}^{-}_{(M)}(S_\gamma)$. To avoid confusion, we always clearly state the domain and range of these maps. 

 We are ready to formulate the main result of this article. 
\begin{theorem}\label{main}
Let $M$ be a strongly regular weight sequence and let $\gamma > 0$. The following statements are equivalent:
\begin{itemize}
\item[$(i)$] $\gamma \leq \gamma(M)$.
\item[$(ii)$] $B:\mathscr{A}^{-}_{(M)}(S_\gamma) \rightarrow \Lambda_{(M)}(\N)$ is surjective. 
\item[$(iii)$] $B: \mathscr{A}^{-}_{(M)}(S_\gamma) \rightarrow \Lambda_{(M)}(\N)$ has a continuous linear right inverse.
 \item[$(iv)$] $B: \mathscr{S}^{-}_{(M)}(S_\gamma) \rightarrow \Lambda_{(M)}(\N)$ is surjective. 
\item[$(v)$] $B: \mathscr{S}^{-}_{(M)}(S_\gamma) \rightarrow \Lambda_{(M)}(\N)$ has a continuous linear right inverse.
\end{itemize}
\end{theorem}
 \begin{remark}
 We are mainly interested in the equivalences of $(i)$--$(iii)$ in Theorem \ref{main} (cf.\ Theorem \ref{intro}). However,  the spaces $\mathscr{S}^{-}_{(M)}(S_\gamma)$ play an essential role in our proof of Theorem \ref{main}; see  Remark \ref{explanation} below for more details.
 \end{remark}
  The proof of Theorem \ref{main} is given in Section \ref{sect-main}.  In preparation for it, we  present various auxiliary results in the next two sections.


\section{Two functional analytic tools}\label{sect-aux1}
\subsection{The Mittag-Leffler procedure}\label{sect-ML} We explain the abstract Mittag-Leffler lemma for Fr\'echet spaces and state it in terms of the derived projective limit functor; we refer to the book \cite{Wengenroth} for more information on this subject. 

A \emph{projective spectrum}  $\mathscr{X} = (X_n, \varrho_{n+1}^n)_{n \in \N}$ consists of vector spaces $X_n$  and linear
maps $\varrho^n_{n+1}: X_{n+1} \rightarrow X_n$ (called the spectral maps). We define $\varrho^n_n  = \operatorname{id}_{X_n}$ and $\varrho^n_{m} = \varrho_{n+1}^n \circ \cdots \circ \varrho^{m-1}_{m} : X_m \rightarrow X_n$ for $n,m \in \N$ with $m > n$. Given two projective spectra $\mathscr{X} = (X_n, \varrho_{n+1}^n)_{n \in \N}$  and $\mathscr{Y} = (Y_n, \sigma_{n+1}^n)_{n \in \N}$, a \emph{morphism} $T = (T_n)_{n \in \N}: \mathscr{X} \rightarrow \mathscr{Y}
$ consists of linear maps $T_n : X_n \rightarrow Y_n$ such that $T_n \circ \varrho_{n+1}^n = \sigma_{n+1}^n \circ T_{n+1}$ for all $n \in \N$. We set
$$
\operatorname{Proj} \mathscr{X} = \{(x_n)_{n \in \N} \in \prod_{n \in \N} X_n \, | \,  x_n = \varrho^n_{n+1}(x_{n+1}),  \forall n \in \N\}
$$
and
$$
\operatorname{Proj} T : \operatorname{Proj} \mathscr{X} \rightarrow \operatorname{Proj} \mathscr{Y}, \, (x_n)_{n \in \N} \mapsto (T_n(x_n))_{n \in \N}.
$$
For each $n \in \N$ we also write $\varrho^n : \operatorname{Proj} \mathscr{X} \rightarrow X_n, \, (x_j)_{j \in \N} \mapsto x_n$. Furthermore, we define
$$ 
\operatorname{Proj}^1 \mathscr{X} =   \prod_{n \in \N}X_n / B(\mathscr{X}),
$$
where
$$B(\mathscr{X}) = \{(x_n)_{n \in \N} \in \prod_{n \in \N} X_n \, | \, \exists \, (u_n)_{n \in \N} \in \prod_{n \in \N} X_n \, : \,  x_n = u_n - \varrho^n_{n+1}(u_{n+1}), \forall n \in \N\}.
$$

Let  
 \begin{center}
\begin{tikzpicture}
  \matrix (m) [matrix of math nodes, row sep=2em, column sep=2.5em]
    {0 & \mathscr{X} & \mathscr{Y} & \mathscr{Z} & 0 \\
   };
  { [start chain] \chainin (m-1-1);
\chainin (m-1-2);
\chainin (m-1-3) [join={node[above,labeled] {S}}];
\chainin (m-1-4)[join={node[above,labeled] {T}}];
\chainin (m-1-5)[join={node[above,labeled] {}}];
}
\end{tikzpicture}
\end{center}
be an exact sequence of projective spectra, i.e., a commutative diagram
\begin{center}
\begin{tikzpicture}
  \matrix (m) [matrix of math nodes, row sep=2em, column sep=2.5em]
    {  
        0 &  X_1 & Y_1  & Z_1 & 0 \\
        0 &  X_{2} & Y_{2}  & Z_{2} & 0 \\
         & \vdots & \vdots  & \vdots &  \\
          };
  
  { [start chain] \chainin (m-1-1);
  \chainin (m-1-2);
\chainin (m-1-3)[join={node[above,labeled] {S_1}}];
\chainin (m-1-4)[join={node[above,labeled] {T_1}}];
\chainin (m-1-5);
; }
  { [start chain] \chainin (m-2-1);
  \chainin (m-2-2);
\chainin (m-2-3)[join={node[above,labeled] {S_{2}}}];
\chainin (m-2-4) [join={node[above,labeled] {T_{2}}}];
\chainin (m-2-5);
; }
  { [start chain] \chainin (m-3-2);
  \chainin (m-2-2);
    \chainin (m-1-2);
; }
  { [start chain] \chainin (m-3-3);
  \chainin (m-2-3);
    \chainin (m-1-3);
; }
  { [start chain] \chainin (m-3-4);
  \chainin (m-2-4);
    \chainin (m-1-4);
; }
  
\end{tikzpicture}
\end{center}
where the vertical arrows stand for the spectral maps and each row is a short exact sequence of vector spaces. If $\operatorname{Proj}^1 \mathscr{X} = 0$, then 
\begin{center}
\begin{tikzpicture}
  \matrix (m) [matrix of math nodes, row sep=2em, column sep=2.5em]
    {0 & \operatorname{Proj}\mathscr{X} & \operatorname{Proj}\mathscr{Y} & \operatorname{Proj}\mathscr{Z} & 0 \\
   };
  { [start chain] \chainin (m-1-1);
\chainin (m-1-2);
\chainin (m-1-3) [join={node[above,labeled] {\operatorname{Proj} T}}];
\chainin (m-1-4)[join={node[above,labeled] {\operatorname{Proj} S }}];
\chainin (m-1-5)[join={node[above,labeled] {}}];
}

\end{tikzpicture}
\end{center}is again exact \cite[Cor.\ 3.1.5]{Wengenroth}.

Two projective spectra  $\mathscr{X} = (X_n, \varrho_{n+1}^n)_{n \in \N}$ and  $\mathscr{Y} = (Y_n, \sigma_{n+1}^n)_{n \in \N}$ are said to be \emph{equivalent} if there are increasing sequences $(k_n)_{n \in \N}$ and $(l_n)_{n \in \N}$ of natural numbers with $n \leq l_n \leq k_n \leq l_{n+1}$ and  linear maps $T_n : X_{k_n} \rightarrow Y_{l_n}$ and $S_n : Y_{l_{n+1}} \rightarrow X_{k_n}$ such that
$S_n \circ T_{n+1} = \varrho^{k_n}_{k_{n+1}}$ and $T_n \circ S_{n} = \sigma^{l_n}_{l_{n+1}}$ for all $n \in \N$. In such a case, we have that  $\operatorname{Proj}^1 \mathscr{X} \cong \operatorname{Proj}^1 \mathscr{Y}$ \cite[Proposition 3.1.7]{Wengenroth}.  

 A \emph{projective spectrum of Fr\'echet spaces}  $\mathscr{X} = (X_n, \varrho_{n+1}^n)_{n \in \N}$ is a projective spectrum consisting of Fr\'echet spaces $X_n$  and continuous
 linear spectral maps $\varrho^n_{n+1}: X_{n+1} \rightarrow X_n$. Given a Fr\'echet space $X$, we denote by $\mathscr{U}_0(X)$ its collection of neighbourhoods of $0$. 
 For projective spectra of Fr\'echet spaces  the vanishing of the derived projective limit functor may be characterized as follows (this result is known as the abstract Mittag-Leffler lemma). 
\begin{theorem}\label{M-L} \cite[Theorem 3.2.8]{Wengenroth} Let $\mathscr{X} = (X_n, \varrho_{n+1}^n)_{n \in \N}$ be a projective spectrum of Fr\'echet spaces. Then, the following statements are equivalent:
\begin{itemize}
\item[$(i)$] $\operatorname{Proj}^1 \mathscr{X} = 0$.
\item[$(ii)$] $\forall n \in \N, \, U \in \mathscr{U}_0(X_n) \, \exists m \geq n \, \forall k \geq m \, : \, \varrho^n_m(X_m) \subset \varrho^n_k(X_k) + U$.
\item[$(iii)$] $\forall n \in \N, \, U \in \mathscr{U}_0(X_n) \, \exists m \geq n  \, : \, \varrho^n_m(X_m) \subset \varrho^n(\operatorname{Proj} \mathscr{X}) + U$. 
\end{itemize} 
\end{theorem}
In what follows we shall not write explicitly the spectral maps of projective spectra as they will be clear from the context. 
\subsection{A splitting theorem} We present a particular version of the $(DN)$-$(\Omega)$ splitting theorem for Fr\'echet spaces; see \cite{Vogt} for more information.

A Fr\'echet space $X$ with a fundamental increasing system $(U_n)_{n \in \N}$ of neighborhoods of zero is said to satisfy property $(\Omega)$ if
$$
\forall n \in \N \, \exists m \geq n  \, \forall k \geq m  \, \exists C,L > 0 \, \forall \varepsilon > 0 \, : U_m \subset  C\varepsilon^{-L}U_k + \varepsilon U_n.
$$
Let $\alpha = (\alpha_j)_{j \in \N}$ be an increasing sequence of positive real numbers with $\lim_{n \to \infty}\alpha_n =  \infty$.  We define the \emph{power series space $\Lambda_\infty(\alpha)$ of infinite type} as the Fr\'echet space consisting of all $(c_j)_{j \in \N} \in \C^\N$ such that
$$
\sum_{j = 0}^\infty |c_j| e^{n \alpha_j} < \infty, \qquad \forall n \in \N.
$$
We then have:
\begin{theorem} \cite[Theorem 5.1]{Vogt}\label{splitting} Let
\begin{center}
\begin{tikzpicture}
  \matrix (m) [matrix of math nodes, row sep=2em, column sep=2.5em]
    {0 & X & Y & Z & 0 \\
   };
  { [start chain] \chainin (m-1-1);
\chainin (m-1-2);
\chainin (m-1-3) [join={node[above,labeled]{ S}}];
\chainin (m-1-4)[join={node[above,labeled] { T}}];
\chainin (m-1-5)[join={node[above,labeled] {}}];
}
\end{tikzpicture}
\end{center}
be a short exact sequence of Fr\'echet spaces with continuous linear maps. If $X$ satisfies $(\Omega)$ and $Z$ is isomorphic to a power series space of infinite type, then $T$ has a continuous linear right inverse. 
\end{theorem}

\section{Spaces of holomorphic functions with rapid decay on strips}\label{sect-aux2}
In this section we introduce a class of weighted Fr\'echet spaces of  holomorphic functions with rapid decay on strips and give two properties of these spaces, which were essentially shown in \cite{Debrouwere}. Moreover,  we establish a connection between these spaces and the spaces $\mathscr{S}^{-}_{(M)}(S_\gamma)$. We point out that the notation employed here is different from the one used in \cite{Debrouwere}.

A continuous increasing function $\nu: [0,\infty) \rightarrow [0,\infty)$ is called a \emph{weight function} if $\nu(0) = 0$ and $\log t = o(\nu(t))$. 
Let $h > 0$. A weight function $\nu$ is said to satisfy condition $(Q_h)$ if there are $C,L >0$ such that
$$
\int^\infty_0\nu(t+s)  e^{-s/ h }ds \leq L\nu(t) + C, \qquad t \geq 0.
$$
We set
$$
h(\nu) := \sup \{ h > 0 \, | \, \  \nu \mbox{ satisfies } (Q_h) \}.
$$ 
 
We define the horizontal strips
$$
V_h := \{ z \in \C \, | \,  | \operatorname{Im} z | < \frac{ h \pi}{2} \}, \qquad h > 0.
$$ 
 Let $\nu$ be a weight function. For $h > 0$ we define $\mathscr{U}_{(\nu)}(V_h)$ as the Fr\'echet space consisting of all $f \in \mathscr{O}(V_h)$  such that
$$
\sup_{z \in V_h} |f(z)|e^{n\nu(| \operatorname{Re} z |)} < \infty, \qquad \forall n \in \N.
$$
The corresponding non-uniform spaces are defined as
$$
\mathscr{U}^{-}_{(\nu)}(V_h) := \varprojlim_{k \to h^{-}}\mathscr{U}_{(\nu)}(V_k).
$$
We  need the following two facts about these spaces.
\begin{theorem}\label{thmstrips}
Let $\nu$ be a weight function and let $0 <h \leq h(\nu)$. Then,
\begin{itemize}
\item[$(i)$] Let $(h_n)_{n \in \N}$ be an increasing sequence of positive real numbers with $h_n < h$ such that $\lim_{n \to \infty} h_n = h$. Consider the projective spectrum $\mathscr{X} = ( \mathscr{U}_{(\nu)}(V_{h_n}))_{n \in \N}$ (with restriction as spectral maps). Then, $\operatorname{Proj}^1 \mathscr{X} = 0$.
\item[$(ii)$] $\mathscr{U}^-_{(\nu)}(V_{h})$ satisfies $(\Omega)$.
\end{itemize}
\end{theorem}
\begin{proof}
$(i)$ We show that $\mathscr{X}$ satisfies condition $(iii)$ from Theorem \ref{M-L}.  For $l > 0$ and $n \in \N$ we define $\mathscr{U}_{\nu,n}(V_l)$ as the Banach space consisting of all $f \in \mathscr{O}(V_l)$  such that
$$
\sup_{z \in V_l} |f(z)|e^{n\nu(| \operatorname{Re} z |)} < \infty.
$$
Consider the projective spectrum $\mathscr{Y} = ( \mathscr{U}_{\nu,n}(V_{h_n}))_{n \in \N}$ (with restriction as spectral maps). Note that $\operatorname{Proj} \mathscr{Y} =  \operatorname{Proj} \mathscr{X} = \mathscr{U}^-_{(\nu)}(V_{h})$. The proof of \cite[Proposition 3.1]{Debrouwere} (see the statement at the beginning of this proof) implies  that $\mathscr{Y}$ satisfies condition $(ii)$ from Theorem \ref{M-L}. Hence,  $\mathscr{Y}$ also satisfies condition $(iii)$ from this result, i.e.,
$$
\forall n \in \N, \, U \in \mathscr{U}_0(\mathscr{U}_{\nu,n}(V_{h_n})) \, \exists m \geq n \, : \, \mathscr{U}_{\nu,m}(V_{h_m}) \subset \mathscr{U}^-_{(\nu)}(V_{h}) + U.
$$
This implies that
$$
\forall n \in \N, \, U \in \mathscr{U}_0(\mathscr{U}_{(\nu)}(V_{h_n})) \, \exists m \geq n \, : \, \mathscr{U}_{(\nu)}(V_{h_m}) \subset \mathscr{U}^-_{(\nu)}(V_{h}) + U,
$$
which means that  $\mathscr{X}$ satisfies condition $(iii)$ from Theorem \ref{M-L}.

$(ii)$ This is shown in \cite[Proposition 3.1]{Debrouwere}.
\end{proof}
Next, we show how the spaces $\mathscr{S}^-_{(M)}(S_\gamma)$ and $\mathscr{U}^-_{(\nu)}(V_h)$ are related to each other. Let $M$ be  a weight sequence and let $\gamma > 0$. We define $\mathscr{S}^{0}_{(M)}(S_\gamma)$ and $\mathscr{S}^{-,0}_{(M)}(S_\gamma)$ as the kernel of the map $B: \mathscr{S}_{(M)}(S_\gamma) \rightarrow  \Lambda_{(M)}(\N)$ and $B: \mathscr{S}^-_{(M)}(S_\gamma) \rightarrow  \Lambda_{(M)}(\N)$, respectively.
\begin{lemma}\label{kernel}
Let $M$ be a strongly regular weight sequence and let $\gamma > 0$. Then, $f \in \mathscr{O}(S_\gamma)$ belongs to  $\mathscr{S}^{-,0}_{(M)}(S_\gamma)$ if and only if 
$$
\| f \|_{n,\delta} = \sup_{z \in S_\delta} |f(z)| e^{n(\omega_M(1/|z|) + \omega_M(|z|))} < \infty, \qquad \forall n \in \N, \delta < \gamma.
$$
Moreover, the topology of $\mathscr{S}^{-,0}_{(M)}(S_\gamma)$ is generated by the family of seminorms $\{ \| \, \cdot \, \|_{n,\delta} \, | \, n \in \N, \delta < \gamma \}$.
\end{lemma}
\begin{proof}
Taylor's theorem and the Cauchy estimates imply that  $f \in \mathscr{O}(S_\gamma)$ belongs to  $\mathscr{S}^{-,0}_{(M)}(S_\gamma)$ if and only if 
$$
| f |_{n, \delta} = \sup_{z \in S_\delta} |f(z)| e^{\omega_M(n/|z|) + \omega_M(n|z|)} < \infty, \qquad \forall n \in \N, \delta < \gamma,
$$
and that the topology of $\mathscr{S}^{-,0}_{(M)}(S_\gamma)$ is generated by the family of seminorms $\{ | \, \cdot \, |_{n,\delta} \, | \, n \in \N, \delta < \gamma \}$.
The result therefore follows from \eqref{assM2} and \eqref{assgamma}.
\end{proof}
Consider the biholomorphic map $\operatorname{Log}: \Sigma \rightarrow \C$ and note that $\operatorname{Log}(S_\gamma) =  V_\gamma$ for all $\gamma > 0$. 
\begin{lemma}\label{equivalence}
Let $M$ be a strongly regular weight sequence  and set $\nu_M(t) = \omega_M(e^t)$ for $t \geq 0$. Then,
\begin{itemize}
\item[$(i)$] $\gamma(M) = h(\nu_M)$.
\item[$(ii)$] For each $\gamma > 0$ the map
$$
\mathscr{U}^-_{(\nu_M)}(V_\gamma) \rightarrow \mathscr{S}^{-,0}_{(M)}(S_\gamma), \, f \mapsto f \circ \operatorname{Log}
$$
is a topological isomorphism.
\end{itemize}
\end{lemma}
\begin{proof}$(i)$ This follows from Lemma \ref{gammaM3}.

$(ii)$ This follows from Lemma \ref{kernel}.
\end{proof}
Lemma \ref{equivalence} enables us to give an analogue of Theorem \ref{thmstrips} for the spaces $\mathscr{S}^{0}_{(M)}(S_\gamma)$ and $\mathscr{S}^{-,0}_{(M)}(S_\gamma)$.
\begin{corollary}\label{proj-omega}
Let $M$ be a strongly regular weight sequence and let $0 < \gamma \leq \gamma(M)$. Then,
\begin{itemize}
\item[$(i)$] Let $(\gamma_n)_{n \in \N}$ be an increasing sequence of positive real numbers with $\gamma_n < \gamma$ such that $\lim_{n \to \infty} \gamma_n = \gamma$. Consider the projective spectrum $\mathscr{X} = (\mathscr{S}^{0}_{(M)}(S_{\gamma_n}))_{n \in \N}$. Then, $\operatorname{Proj}^1 \mathscr{X} = 0$.
\item[$(ii)$] $\mathscr{S}^{-,0}_{(M)}(S_\gamma)$ satisfies $(\Omega)$.
\end{itemize}
\end{corollary}
\begin{proof}
$(i)$ Consider the projective spectra 
$$\mathscr{Y} = (\mathscr{S}^{-,0}_{(M)}(S_{\gamma_n}))_{n \in \N}, \qquad  \mathscr{Z} = (\mathscr{U}^{0}_{(\nu_M)}(V_{\gamma_n}))_{n \in \N}.
$$
 It is clear that $\mathscr{X}$ and $\mathscr{Y}$ are equivalent, while Lemma \ref{equivalence}$(ii)$ implies that $\mathscr{Y}$ and $\mathscr{Z}$ are equivalent. Hence, by Theorem \ref{thmstrips}$(i)$ and Lemma \ref{equivalence}$(i)$, $\operatorname{Proj}^1\mathscr{X} \cong \operatorname{Proj}^1\mathscr{Y} \cong \operatorname{Proj}^1\mathscr{Z} = 0$ (cf.\ Subsection \ref{sect-ML}).

$(ii)$ This follows from Theorem \ref{thmstrips}$(ii)$ and Lemma \ref{equivalence}.
\end{proof}

\begin{remark}\label{explanation}
Corollary \ref{proj-omega} is the reason why we also consider the spaces $\mathscr{S}^{-}_{(M)}(S_\gamma)$ in Theorem \ref{main}. We do not know whether a result similar to Corollary \ref{proj-omega} holds for the spaces $\mathscr{A}^{-}_{(M)}(S_\gamma)$.
\end{remark}
\section{Proof of the main result}\label{sect-main}
This section is devoted to the proof of Theorem \ref{main}. We start by introducing an auxiliary space of ultraholomorphic functions. Let $M$ be a weight sequence and let $\gamma > 0$. A function $f \in \mathscr{O}(S_\gamma)$ is said to admit the formal power series $\sum_{p= 0}^\infty a_p z^p$, $a_p \in \C$, as its \emph{uniform (M)-asymptotic expansion in $S_\gamma$} if for all $L > 0$ 
$$
 \sup_{p \in \N} \sup_{z \in S_\gamma} \frac{L^p}{M_p|z|^p} \left | f(z) - \sum_{q=0}^{p-1} a_q z^q \right| < \infty.
$$
We denote by $\widetilde{\mathscr{A}}_{(M)}(S_\gamma)$  the space consisting of all $f \in \mathscr{O}(S_\gamma)$ that admit a uniform $(M)$-asymptotic expansion  in $S_\gamma$.
Taylor's theorem and the Cauchy estimates imply that
$$
\mathscr{A}_{(M)}(S_\gamma) \subset \widetilde{\mathscr{A}}_{(M)}(S_\gamma) \subset \mathscr{A}^-_{(M)}(S_{\gamma})
$$
and that $a_p =f^{(p)}(0)/p!$, $p \in \N$, for all $f \in \widetilde{\mathscr{A}}_{(M)}(S_\gamma)$, where $\sum_{p= 0}^\infty a_p z^p$ is a uniform (M)-asymptotic expansion of $f$ in $S_\gamma$. In particular, the coefficients $a_p$ are unique.   Our interest in the spaces  $\widetilde{\mathscr{A}}_{(M)}(S_\gamma)$ stems from the following observation (cf.\ the proof of \cite[Theorem 4.17]{JG-S-S}).
\begin{lemma} \label{ramification}
Let $M$ be a strongly regular weight sequence, let $n \in \Z_+$, and let $\gamma > 0$. Suppose that $g \in \widetilde{\mathscr{A}}_{(M^{1/n})}(S_{\gamma/n})$ is such that $g^{(p)}(0) = 0$ for all $p \in \N \backslash n \N$. Define $f(z) = g(z^{1/n})$ for $z \in S_\gamma$. Then, $f \in \widetilde{\mathscr{A}}_{(M)}(S_{\gamma})$ and $f^{(p)}(0)/p! = g^{(pn)}(0)/(pn)!$ for all $p \in \N$.
\end{lemma}
\begin{proof}
It is clear that $f \in \mathscr{O}(S_\gamma)$. The assumptions on $g$ imply that  \begin{align*}
&\sup_{z \in S_\gamma}\left | f(z) - \sum_{q=0}^{p-1}  \frac{g^{(qn)}(0)}{(qn)!} z^q \right| =  \sup_{w \in S_{\gamma/n}}\left | g(w) - \sum_{q=0}^{p-1}  \frac{g^{(qn)}(0)}{(qn)!} z^{nq} \right| \\
&=  \sup_{w \in S_{\gamma/n}}\left | g(w) - \sum_{j=0}^{pn-1}  \frac{g^{(j)}(0)}{(j)!} z^{j} \right| \leq CL^{-p}M^{1/n}_{pn}, \qquad p \in \N,
\end{align*}
for all $L > 0$ and suitable $C = C_L >0$.  The result now follows from the fact that, by $(M.2)$,
$$
M^{1/n}_{pn} \leq CH^pM_p, \qquad p \in \N,
$$
for some $C,H > 0$.
\end{proof}

We need the following improvement of  Theorem \ref{theorem-T}.
\begin{proposition}\label{surjective-S}
Let $M$ be a strongly regular weight sequence and let $0 < \gamma < \gamma(M)$. Then, $B: \mathscr{S}_{(M)}(S_\gamma) \rightarrow \Lambda_{(M)}(\N)$ is surjective.
\end{proposition}
%
\begin{proof}
We distinguish two cases (cf.\ the proof of \cite[Theorem 3.2.1]{Thilliez}).

CASE I: $\gamma < 2$. In this case, we may view $S_\gamma$ as a subset of the complex plane $\C$. Choose $\gamma < \delta < \gamma(M)$ and set $N = M^{\delta/\gamma(M)}$. Then, $\gamma(N) = \delta$ and $N \prec M$. By Proposition \cite[Theorem 2.3.1]{Thilliez} (applied to $N$), there exists $G \in \mathscr{O}(S_\gamma)$ such that $G$ does not vanish on $S_\gamma$ and 
$$
|G(z)| \leq e^{-\omega_N(\kappa/|z|)}, \qquad z \in S_\gamma,
$$
for some $\kappa > 0$. Since $N \prec M$, we have that
$$
|G(z)| \leq Ce^{-\omega_M(L/|z|)}, \qquad z \in S_\gamma,
$$
for all $L > 0$ and suitable $C = C_L > 0$.  Set $R_\gamma = \{ z \in \C \, | \, z+1 \in S_\gamma \}$ and define $P(z) = G(1/(z+1))$ for $z \in R_\gamma$. Then, $P$ does not vanish on $R_\gamma$ and
\begin{equation}
|P(z)| \leq Ce^{-\omega_M(L|z|)}, \qquad z \in R_\gamma,
\label{bounds-H}
\end{equation}
for all $L > 0$ and suitable $C = C_L > 0$. Choose $\varepsilon > 0$ such that for  each $z \in S_\gamma$ the closed disc with center $z$ and radius $\varepsilon$ belongs to $R_\gamma$. The Cauchy estimates for $P$ and \eqref{bounds-H} imply that
\begin{equation}
|P^{(p)}(z)| \leq C\varepsilon^{-p}p! e^{-\omega_M(L|z|)}, \qquad p\in \N,z \in S_\gamma, 
\label{bounds-Hder}
\end{equation}
for all $L > 0$ and suitable $C =C_L > 0$. 
As $P$ does not vanish on $R_\gamma$, the Cauchy estimates for $1/P$ yield that
\begin{equation}
|(1/P)^{(p)}(0)| \leq C\varepsilon^{-p}p!, \qquad p\in \N,
\label{inverse}
\end{equation}
for some $C > 0$. Now let $a =(a_p)_{p \in \N} \in \Lambda_{(M)}(\N)$ be arbitrary. For $p \in \N$ we define
$$
b_p =  \sum_{j=0}^p \binom{p}{j} a_j (1/P)^{(p-j)}(0). 
$$
Inequality \eqref{inverse} implies that $b= (b_p)_{p \in \N}   \in  \Lambda_{(M)}(\N)$. By  Theorem \ref{theorem-T}, there exists $g \in \mathscr{A}_{(M)}(S_\gamma)$ such that $B(g) = b$. Set $f = gP$ and note that, by \eqref{bounds-Hder}, $f \in \mathscr{S}_{(M)}(S_\gamma)$. Finally, it is clear that $B(f) = a$.

CASE II: $\gamma \geq 2$. Choose $\gamma < \delta < \gamma(M)$ and $n \in \Z_+$ such that $\delta/n < 2$. Let $a = (a_p)_{p \in \N} \in \Lambda_{(M)}(\N)$ be arbitrary.  Consider $b = (b_p)_{p \in \N} \in \C^\N$, where
	\[b_p = \begin{cases}
	\frac{(nj)!}{j!}a_j,&  p = nj \mbox{ for some $j \in \N$}, \\
	0, & \mbox{otherwise}.\\
	\end{cases}\]
Condition $(M.1)$ implies that $b \in 	 \Lambda_{(M^{1/n})}(\N)$. Since $\delta/n < 2$ and $\delta/n < \gamma(M)/n = \gamma(M^{1/n})$, CASE I (applied to $M^{1/n}$) implies that there exists $g \in \mathscr{S}_{(M^{1/n})}(S_{\delta/n})$ such that $B(g) = b$. Define $f(z) = g(z^{1/n})$ for $z \in S_{\delta} $. By Lemma \ref{ramification}, we have that $f \in \widetilde{\mathscr{A}}_{(M)}(S_{\delta})$ and $B(f) = a$. We now show that $f \in \mathscr{S}_{(M)}(S_{\gamma})$. We have that $f \in \widetilde{\mathscr{A}}_{(M)}(S_{\delta}) \subset \mathscr{A}^-_{(M)}(S_{\delta})$. Since $g \in \mathscr{S}_{(M^{1/n})}(S_{\delta/n})$, the equality $\omega_{M^{1/n}}(t) = \omega_M(t^n)/n$ for $t \geq 0$ and \eqref{assM2} yield that
$$
|f(z)| = |g(z^{1/n})| \leq Ce^{-\omega_M(L|z|)}, \qquad z \in S_\delta,
$$
for all $L > 0$ and suitable $C = C_L > 0$. Hence,  the Cauchy estimates imply that $f \in \mathscr{S}_{(M)}(S_{\gamma})$.
\end{proof}
The next result will be used to prove the implication $(ii) \Rightarrow (i)$ in Theorem \ref{main}.
\begin{proposition} \label{necc}
Let $n \in \Z_+$ and suppose that $B: \widetilde{\mathscr{A}}_{(M)}(S_n) \rightarrow \Lambda_{(M)}(\N)$ is surjective. Then, $\gamma(M) > n$.
\end{proposition}
Proposition \ref{necc} is essentially shown in \cite{S-V} (see also \cite[Theorem 4.14$(i)$]{JG-S-S} for the Roumieu variant of this result). We briefly indicate the main idea of the proof for the reader's convenience. Let $M$ be a weight sequence and let $n \in \Z_+$. We denote by $\mathscr{N}_{(M),n}([0,\infty))$ the space consisting of all $f \in C^\infty([0,\infty))$ such that $f^{(p)}(0) = 0$ for all $p \in \N \backslash n \N$ and 
$$
\sup_{p \in \N} \sup_{x \geq 0} \frac{L^p|f^{(np)}(x)|}{p!M_p} < \infty
$$
for all $L > 0$. Note that  this definition differs from the one used in  \cite{S-V}.
\begin{proof}[Proof of Proposition \ref{necc}] We define $M^* = (M_p/p!)_{p \in \N}$. In \cite[Proposition 4.3]{S-V} it is shown that the surjectivity of the map
\begin{equation}
\mathscr{N}_{(M^*),n}([0,\infty))  \rightarrow \Lambda_{(M^*)}(\N), \, f \mapsto (f^{(np)}(0))_{p \in \N}
\label{borel}
\end{equation}
implies that  $M^{1/n}$ satisfies $(M.3)$. Hence, by Lemma \ref{gammaM3}, it suffices to show that  the map  \eqref{borel} is surjective if the map $B: \widetilde{\mathscr{A}}_{(M)}(S_n) \rightarrow \Lambda_{(M)}(\N)$ is so. In the Roumieu case, this is done in \cite[Theorem 4.14$(i)$]{JG-S-S} and the same proof applies to the Beurling case (see also the proof of \cite[Theorem 4.6]{S-V}). We point out that the asymptotic Borel map and the sequence space associated to $M$ employed here are different from the ones used in  \cite{JG-S-S}.
\end{proof}
\begin{proof}[Proof of Theorem \ref{main}]
We  show the chain of implications $(i) \Rightarrow (iv) \Rightarrow (v) \Rightarrow (iii) \Rightarrow (ii) \Rightarrow (i)$. The implications $(v) \Rightarrow (iii)$ and $(iii) \Rightarrow (ii)$ are trivial. We now prove the three other ones.

$(i) \Rightarrow (iv)$ Let $(\gamma_n)_{n \in \N}$ be an increasing sequence of positive real numbers with $\gamma_n < \gamma$ such that $\lim_{n \to \infty} \gamma_n = \gamma$. We set $B_n = B:
\mathscr{S}_{(M)}(S_{\gamma_n}) \rightarrow \Lambda_{(M)}(\N)$ for $n \in \N$. Consider the projective spectra
$$
\mathscr{X} = (\mathscr{S}^{0}_{(M)}(S_{\gamma_n}))_{n \in \N}, \qquad \mathscr{Y} = (\mathscr{S}_{(M)}(S_{\gamma_n}))_{n \in \N}, \qquad \mathscr{Z} = (\Lambda_{(M)}(\N))_{n \in \N},
$$
and note that $(B_n)_{n \in \N} : \mathscr{Y} \rightarrow \mathscr{Z}$ is a morphism. We need to show that
$$
B = \operatorname{Proj} (B_n)_{n \in \N} : \mathscr{S}^{-}_{(M)}(S_{\gamma}) =  \operatorname{Proj} \mathscr{Y}  \rightarrow  \Lambda_{(M)}(\N) =  \operatorname{Proj} \mathscr{Z}
$$
is surjective. Proposition \ref{surjective-S} yields that
\begin{center}
\begin{tikzpicture}
  \matrix (m) [matrix of math nodes, row sep=2em, column sep=3.5em]
    {0 & \mathscr{X}  & \mathscr{Y}  & \mathscr{Z}  & 0 \\
   };
  { [start chain] \chainin (m-1-1);
\chainin (m-1-2);
\chainin (m-1-3) [join={node[above,labeled]{ }}];
\chainin (m-1-4)[join={node[above,labeled] { (B_n)_{n \in \N}}}];
\chainin (m-1-5)[join={node[above,labeled] {}}];
}
\end{tikzpicture}
\end{center}
is a short exact sequence of projective spectra. As explained in Subsection \ref{sect-ML}, the result therefore follows from the fact that $\operatorname{Proj}^1 \mathscr{X} = 0$ (Corollary \ref{proj-omega}$(i)$).

$(i) \Rightarrow (iv)$ The assumption yields that
\begin{center}
\begin{tikzpicture}
  \matrix (m) [matrix of math nodes, row sep=2em, column sep=2.5em]
    {0 & \mathscr{S}^{-,0}_{(M)}(S_{\gamma})   & \mathscr{S}^{-}_{(M)}(S_{\gamma})   & \Lambda_{(M)}(\N)  & 0 \\
   };
  { [start chain] \chainin (m-1-1);
\chainin (m-1-2);
\chainin (m-1-3) [join={node[above,labeled]{ }}];
\chainin (m-1-4)[join={node[above,labeled] { B}}];
\chainin (m-1-5)[join={node[above,labeled] {}}];
}
\end{tikzpicture}
\end{center}
is a short exact sequence of Fr\'echet spaces  with continuous linear maps. Since $\mathscr{S}^{-,0}_{(M)}(S_{\gamma})$ satisfies $(\Omega)$ (Corollary \ref{proj-omega}$(ii)$) and $\Lambda_{(M)}(\N) \cong \Lambda_\infty(j)$, the result is a consequence of Theorem \ref{splitting}.

$(ii) \Rightarrow (i)$ It suffices to show that for all $\delta \in \mathbb{Q}$ with $0 < \delta < \gamma$ it holds that $\delta < \gamma(M)$. Write $ \delta= n/m$ with $n,m \in \Z_+$.  We claim that $B :  \widetilde{\mathscr{A}}_{(M^m)}(S_{n}) \rightarrow  \Lambda_{(M^m)}(\N)$ is surjective. Then, by Proposition \ref{necc}, we would have that $\gamma(M^m) > n$ and, thus, $\gamma(M) > n/m =  \delta$. We now show the claim. Let $a = (a_p)_{p \in \N} \in \Lambda_{(M^m)}(\N)$ be arbitrary.  For $p \in \N$ we define
	\[b_p = \begin{cases}
	\frac{(mj)!}{j!}a_j,&  p = mj \mbox{ for some $j \in \N$}, \\
	0, & \mbox{otherwise}.\\
	\end{cases}\]
Condition $(M.1)$ implies that $b  = (b_p)_{p \in \N} \in  \Lambda_{(M)}(\N)$. The assumption yields that there exist $g \in \mathscr{A}^-_{(M)}(S_{\gamma}) \widetilde{\mathscr{A}}_{(M)}(S_{\delta})$ such that $B(g) = b$. Define $f(z) = g(z^{1/m})$ for $z \in S_{\delta m} = S_n$. By Lemma \ref{ramification}, we have that $f \in \widetilde{\mathscr{A}}_{(M^m)}(S_{n})$ and $B(f) = a$.
\end{proof} 

\end{document}